\documentclass[11pt]{amsart}

\usepackage{amssymb}
\usepackage[all]{xy}

\usepackage{xcolor}
\usepackage{pdfsync}
\definecolor{rltblue}{rgb}{0,0,0.4}
\definecolor{drkred}{rgb}{0.6,0,0}
\definecolor{drkgreen}{rgb}{0,0.4,0}
\usepackage{graphicx}
\usepackage[mathscr]{eucal} 

\usepackage{todonotes}

\usepackage{lmodern}
\usepackage{amssymb,amsmath,amsthm}
\usepackage{mathtools, MnSymbol}
\usepackage{setspace}
\usepackage{tikz-cd}
\usepackage[T1]{fontenc}
\usepackage[utf8]{inputenc}
\usepackage{textcomp} 
\usepackage{stmaryrd}
\usepackage{datetime}
\usepackage{todonotes}
\usepackage{mdframed}
\usepackage[all]{xy}
\RequirePackage{tabularray}

\usepackage{thmtools} 
\usepackage{thm-restate}

\usepackage[colorlinks=true, urlcolor=rltblue, citecolor=drkgreen, linkcolor=drkred] {hyperref}

\usepackage[capitalize]{cleveref}

\newcommand{\bfSigma}{\mathbf{\Sigma}}

\newcommand{\bfPi}{\mathbf{\Pi}}

\usepackage{enumitem}

\newcommand{\A}{\mathcal{A}}
\newcommand{\B}{\mathcal{B}}

\newcommand{\N}{\mathcal{N}}

\newcommand{\SR}{\text{SR}}
\newcommand{\Z}{\mathbb{Z}}

\renewcommand{\L}{\mathcal{L}}
\newcommand{\mc}[1]{\mathcal{#1}}

\newcommand{\Pinf}[1]{\Pi_{#1}}
\newcommand{\Sinf}[1]{\Sigma_{#1}}

\newcommand{\dSinf}[1]{d\text{-}\Sigma_{#1}}

\renewcommand{\phi}{\varphi}

\DeclareMathOperator{\fin}{fin}
\newcommand{\Efin}{\mathrel{E_{\fin}} }

\newcommand{\EZ}{\mathrel{E_{\mathbb{Z}}} }

\def\hbar{{\bar{h}}}

\def\A{\mathcal A}
\def\B{\mathcal{B}}

\def\L{\mathcal L}

\def\N{{\mathcal N}}

\newcommand{\Res}[2]{{#1}^{\leq #2}}

\newtheorem{theorem}{Theorem}[section]
\newtheorem{lemma}[theorem]{Lemma}
\newtheorem{proposition}[theorem]{Proposition}

\theoremstyle{definition}
\newtheorem{definition}[theorem]{Definition}

\theoremstyle{remark}

\newtheorem{example}[theorem]{Example}

\newtheorem*{notation}{Notation}

\newtheorem{claim}{Claim}[theorem]
\newtheorem{question}{Question}

\def\and{\mathrel{\&}}

\begin{document}
	
	\title{Scott Analysis, Linear Orders and almost periodic Functions}
	\author{David Gonzalez \and Matthew Harrison-Trainor \and Meng-Che ``Turbo'' Ho}
	\thanks{The second author acknowledges support from the National Science Foundation under Grant No.\ \mbox{DMS-2153823}. The third author acknowledges support from the National Science Foundation under Grant No.\ \mbox{DMS-2054558}.}
	
	\begin{abstract}
		For any limit ordinal $\lambda$, we construct a linear order $L_\lambda$ whose Scott complexity is $\Sigma_{\lambda+1}$.
		This completes the classification of the possible Scott sentence complexities of linear orderings. Previously, there was only one known construction of any structure (of any signature) with Scott complexity $\Sigma_{\lambda+1}$, and our construction gives new examples, e.g., rigid structures, of this complexity.
		
		Moreover, we can construct the linear orders $L_\lambda$ so that not only does $L_\lambda$ have Scott complexity $\Sigma_{\lambda+1}$, but there are continuum-many structures $M \equiv_\lambda L_\lambda$ and all such structures also have Scott complexity $\Sigma_{\lambda+1}$. In contrast, we demonstrate that there is no structure (of any signature) with Scott complexity $\Pi_{\lambda+1}$ that is only $\lambda$-equivalent to structures with Scott complexity $\Pi_{\lambda+1}$.
		
		Our construction is based on functions $f \colon \mathbb{Z}\to \mathbb{N}\cup \{\infty\}$ which are almost periodic but not periodic, such as those arising from shifts of the $p$-adic valuations.
	\end{abstract}
	
	\maketitle

	\section{Introduction}
	
	Scott \cite{Sc65} proved that for every countable structure $\mc{A}$, there is a sentence $\varphi$ of the infinitary logic $\mc{L}_{\omega_1 \omega}$ that describes $\mc{A}$ up to isomorphism among countable structures in the sense that for any countable structure $\mc{B}$, $\mc{B} \cong \mc{A}$ if and only if $\mc{B} \models \varphi$. Such a sentence is called a Scott sentence for $\mc{A}$. To each structure $\mc{A}$ we can assign an ordinal-valued Scott rank which measures the least complexity of a Scott sentence for $\mc{A}$, essentially, the complexity of describing $\mc{A}$ up to isomorphism. There are a number of different non-equivalent definitions of Scott rank which are technically different but coarsely the same. The two with the tightest connection to Scott sentences were introduced by Montalb\'an in \cite{Mo15}, but the best reference is in Montalb\'an's book \cite{Mo21}. The parameterless Scott rank, $uSR(\mc{A})$, is the least $\alpha$ such that $\mc{A}$ has a $\Pi_{\alpha+1}$ Scott sentence, and the (parametrized) Scott rank, $pSR(\mc{A})$, is the least $\alpha$ such that $\mc{A}$ has a $\Sigma_{\alpha+2}$ Scott sentence. We will use $SR(\mc{A})$ in situations where it does not matter which version of Scott rank we choose, e.g., we will write $SR(\mc{A}) < \lambda$ with $\lambda$ a limit ordinal. There are also internal characterizations in terms of the definability of orbits and several other equivalencies \cite{Mo15}.
	
	Alvir, Greenberg, Harrison-Trainor, and Turetsky \cite{AGHT} defined the Scott complexity of a structure $\mc{A}$ as the least Wadge degree of the set of isomorphic copies of $\mc{A}$. The only possibilities are $\bfSigma^0_\alpha$, $\mathbf{d-}\bfSigma^0_\alpha$, and $\bfPi^0_\alpha$, and so by the Lopez-Escobar theorem \cite{Lo65}, the Scott complexity of $\mc{A}$ can be equivalently defined as follows:
	
	\begin{definition}
		The \emph{Scott complexity} of a structure $\mc{A}$, denoted as $SSC(\mc{A})$, is the least complexity, from among $\Sigma_\alpha$, d-$\Sigma_\alpha$, and $\Pi_\alpha$, of a Scott sentence for $\mc{A}$.
	\end{definition}
	\noindent This definition generalizes and is finer than both the parametrized and parameterless Scott ranks.
	
	Certain Scott complexities are impossible \cite{AGHT}. For example, any structure with a $\Sigma_\lambda$ Scott sentence ($\lambda$ a limit ordinal) must have a sentence of lower complexity, namely, one of the disjuncts. $\Sigma_2$ is also not a possible Scott complexity. After eliminating such examples, older work of Miller \cite{Mi83} gave examples of all of the remaining possible Scott complexities except for $\Sigma_{\lambda + 1}$ with $\lambda$ a limit ordinal. In \cite{AGHT}, such examples were given using a new technique developed by Turetsky.\footnote{This clever technique was also used for several other results in \cite{Turetsky20}. The idea is that the structure represents the vertices of the infinite-dimensional hypercube (with finite support), and the automorphisms are generated by the reflections. Fixing any single element/vertex makes the structure rigid.}
	\begin{theorem}[Alvir, Greenberg, Harrison-Trainor, and Turetsky \cite{AGHT}]
		The possible Scott complexities of countable structures are:
		\begin{enumerate}
			\item $\Pi_\alpha$ for $\alpha \geq 1$;
			\item $\Sigma_\alpha$ for $\alpha \geq 3$ a successor ordinal;
			\item d-$\Sigma_\alpha$ for $\alpha \geq 1$ a successor ordinal.
		\end{enumerate}
		There is a countable structure with each of these Scott complexities.
	\end{theorem}
	
	The exact relationships between Scott complexity, parameterless Scott rank, and parameterized Scott rank were first understood in  \cite{AGHT}.
	Table \ref{table:invariants} taken from \cite{Mo21} Chapter II.7 shows how this relationship works.

	\vspace{.1 in}
	
	\begin{table}
		\centering
		\begin{tblr}
			{
				hlines, vlines,
				hline{1,Z} = {2-3}{solid},
				hline{2} = {4}{solid},
				hline{3} = {6-7}{solid},
				colspec={|c|c|c|c|},
			}
			SSC                &    pSR   & uSR          & complexity of parameters \\
			$\Sinf{\alpha+2}$  & $\alpha$ & $\alpha+2$ & $\Pinf{\alpha+1}$ \\
			$\dSinf{\alpha+1}$ & $\alpha$ & $\alpha+1$ & $\Pinf{\alpha}$ \\
			$\Pinf{\alpha+1}$  & $\alpha$ & $\alpha$   & none \\
			\SetCell[r=1,c=4]{c} $\alpha$ limit &&&\\
			$\Sinf{\alpha+1}$  & $\alpha$ & $\alpha+1$ & $\Pinf{\alpha}$ \\
			$\Pinf{\alpha}$    & $\alpha$ & $\alpha$   & none \\
		\end{tblr}
		
		\caption{Relationship of the different Scott invariants.
			The last column contains the complexity of the automorphism orbit of the
			parameters involved in the parameterized Scott rank.}
		\label{table:invariants}
	\end{table}
	\vspace{.1 in}

	Gonzalez and Rossegger \cite{RosseggerGonzalez} recently gave an almost-complete classification of the Scott ranks of linear orders. Linear orders are a particularly interesting case because they are one of the known examples in which Vaught's conjecture holds \cite{St78}, and the proof of this fact uses Scott rank. Moreover, Gonzalez and Rossegger showed that all of the possible Scott complexities can be achieved by a linear order, except for $\Sigma_3$, which they show is impossible, and $\Sigma_{\lambda + 1}$ for $\lambda$ a limit ordinal, which they leave open \cite[Question 3.13]{RosseggerGonzalez}. This is the same case that was originally left unresolved by Miller.
	
	In the first half of the paper, we close this last gap by showing that for every limit ordinal $\lambda$, there is a linear order of Scott complexity $\Sigma_{\lambda+1}$.
	
	\begin{theorem}\label{thm:main}
		For each limit ordinal $\lambda$ there is a linear order of Scott complexity $\Sigma_{\lambda+1}$.
	\end{theorem}
	
	\noindent Thus, all of the possible Scott complexities for arbitrary structures are Scott complexities of linear orders, except for $\Sigma_3$. This example also gives the second known way of constructing a structure of Scott complexity $\Sigma_{\lambda+1}$ after the example described above, and the first known example in a natural class of non-universal structures.
	
	The proof of Theorem \ref{thm:main} is a construction that takes as input a certain sequence of linear orders $(L_i)_{i \in \omega}$ whose Scott ranks approach $\lambda$ from below, and an appropriate delimiting linear order $K$, and produces a linear order $L_\lambda$ with Scott complexity $\Sigma_{\lambda + 1}$. By applying this construction to different sequences and delimiters, we get the following particular examples:
	
	\begin{restatable}{corollary}{rigid}
		\label{thm:rigid}
		For each limit ordinal $\lambda$ there is a rigid linear order of Scott complexity $\Sigma_{\lambda + 1}$.
	\end{restatable}
	
	\noindent This is the first known example of a rigid structure with this Scott complexity.
	
	\begin{restatable}{corollary}{scattered}
		\label{thm:scattered}
		For each limit ordinal $\lambda$ there is a scattered linear order of Scott complexity $\Sigma_{\lambda+1}$.
	\end{restatable}
	
	\noindent This is the first known example of a scattered linear order with Scott complexity $\Sigma_\alpha$ for any $\alpha$, not just the successor of a limit ordinal. We also show in Proposition \ref{prop:no-scatter-sigma4} that a scattered linear order cannot have Scott complexity $\Sigma_4$.
	
	In the second half of the paper, we analyze the above construction closely to demonstrate a surprising abundance of Scott complexity  $\Sigma_{\lambda+1}$ structures.
	In particular, we produce continuum many Scott complexity $\Sigma_{\lambda + 1}$ structures which form a $\equiv_\lambda$-equivalence class.
	
	\begin{restatable}{theorem}{maintwo}\label{thm:main2}
		For each limit ordinal $\lambda$, there is a linear ordering $L_\lambda$ whose $\equiv_\lambda$-equivalence class contains continuum many non-isomorphic models and they are all of Scott complexity $\Sigma_{\lambda+1}$.
	\end{restatable}
	
	\noindent Thus, while one might have expected that there are few examples of structures of Scott complexity $\Sigma_{\lambda+1}$, we have in fact shown that there are many such examples.
	The above theorem represents a new phenomenon even for general structures and even if the $\lambda$ above is replaced by any non-limit ordinal $\alpha$.
	
	In fact, it demonstrates two novel properties simultaneously.
	Firstly, it gives continuum many Scott complexity $\Sigma_{\alpha+1}$ structures that are all maximally similar (i.e. related by $\equiv_\alpha$).
	We are not aware of any prior example in the literature of even a pair of structures $M_\alpha\equiv_\alpha N_\alpha$ with Scott complexities $\Sigma_{\alpha+1}$.
	Secondly, it gives an example of a $\Pi_{\lambda}$ sentence (namely the one characterizing the $\equiv_\lambda$ theory of $L_\lambda$) with the property that all models of that sentence are Scott complexity $\Sigma_{\lambda+1}$, and thus none of them have a $\Pi_{\alpha+1}$ Scott sentence.
	
	Montalb\'an asked at the 2013 BIRS Workshop on Computable Model Theory whether every $\Pi_\alpha$ sentence had a model with a $\Pi_{\alpha+1}$ Scott sentence; at the time, it seemed possible (perhaps even likely) that this was true.	
	Previously, Harrison-Trainor \cite{HT18} was able to carefully produce, for any given $\alpha$, a $\Pi_2$ sentence all of whose model are Scott rank\footnote{The notion of Scott rank used in \cite{HT18} was one defined via the stabilization of the back-and-forth relations $\equiv_\alpha$.} $\alpha$.
	Thus, though we are not the first to answer Montalb\'an's question, we give a new counterexample.
	To the best knowledge of the authors, Harrison-Trainor's work in \cite{HT18} was the only known example giving an answer to Montalb\'an's question until now.
	It is surprising then, to find an example of this phenomenon outside of a tailor-made solution and in a natural class like linear orderings.
	Furthermore, Harrison-Trainor's construction does not control the structures to the level of Scott complexity and it is not difficult to confirm that his produced example of Scott rank $\lambda+1$ is not of complexity $\Sigma_{\lambda+1}$.
	
	In contrast to the above result, we show that there is no way to replicate the above construction with $\Pi$ sentences. Note that this covers all structures, not just linear orderings. 
	
	\begin{restatable}{theorem}{nopistr}
		There is no limit ordinal $\lambda$ and structure $L_\lambda$ such that every $M$ with $M\equiv_\lambda L_\lambda$ has Scott complexity $\Pi_{\lambda+1}$.
	\end{restatable}
	
	\noindent In particular, this means that the methods of Harrison-Trainor in \cite{HT18} cannot be directly adapted to the analogous problem for Scott complexity.

	\begin{notation}
		All of our structures are assumed to be countable. We write $\omega, \zeta, \eta$ to mean the order types of natural numbers, integers, and rational numbers, respectively.
	\end{notation}
	
	\section{The construction of linear orders with Scott complexity $\Sigma_{\lambda+1}$}
	
	In this section we define a family of linear orderings that we will later show have Scott complexity $\Sigma_{\lambda+1}$.
	We also prove some basic properties about the objects that we define.
	We begin by defining an important family of functions on the integers.

	\begin{definition}
		A function $f\colon\Z\to\mathbb{N} \cup \{\infty\}$ is \emph{almost periodic} if it is not periodic yet for all $n\in\mathbb{N}$ the function $\Res{f}{n}:=\min(n,f)$ is periodic.\footnote{We chose the name \textit{almost periodic} because the examples we produce using the $p$-adics are almost periodic in the sense of Bohr \cite{Bohr}, which is that for every $\epsilon > 0$ the function is periodic up to error $\epsilon$. We include the requirement that an almost periodic function be non-periodic for brevity because these are the only examples we will be interested in.}
	\end{definition}
	
	There are many examples of almost periodic functions. Perhaps the simplest is the $p$-adic valuation $f(k) = v_p(k)$. This takes on one infinite value at $k = 0$.
	
	There are also examples that output only finite values, such as a non-integer shift of the $p$-adic valuation $f(k) = v_p(k -t)$ for $t \in \mathbb{Z}_p - \mathbb{Z}$. We say that such a function \textit{takes on only finite values}. We will check that these functions are almost periodic in Section \ref{sec:many}, and show that there are continuum-many non-shift-equivalent examples.
	
	
	\subsection{Almost periodic functions taking on only finite values}
	
	For now, consider only almost periodic functions taking on only finite values.
	
	The following definition is key in the construction of a $\Sigma_{\lambda+1}$ linear ordering for a limit ordinal $\lambda$. Unless otherwise noted, intervals mean open intervals, namely, sets of the form $(a,b) = \{x \mid a < x < b\}$ for some $a,b\in L$. 
	
	\begin{definition}\label{defn:l-mixable}
		An ordered pair $(\{L_i\}_{i\in\omega}, K)$ of a sequence of linear orderings $\{L_i\}_{i\in\omega}$ and a single linear ordering $K$ is called a \emph{$\lambda$-mixable pair} if the following properties hold for some  non-zero fundamental sequence $\delta_i\to\lambda$:
		\begin{enumerate}
			\item $\SR(K)<\lambda$
			\item $\SR(L_i)<\lambda$
			\item $L_i\equiv_{\delta_i} L_{i+1}$
			\item $L_i\not\equiv_{\delta_{i+1}} L_{i+1}$
			\item any finite alternating sum $1+L_{a_0}+1+K+1+L_{a_1}+1+K+\cdots+1+L_{a_n}$ has intervals isomorphic to $K$ only within the written $K$ blocks (or as the entire written $K$ block).
		\end{enumerate}
	\end{definition}
	
	Generally speaking, the first two criteria are not too difficult to produce examples of, as Scott rank is a fairly well studied concept in the context of linear orderings.
	The final criterion allows the copies of $K$ to act as delimiters.
	The third and fourth criteria are at the core of the definition.
	These guarantee that these structures resemble $\lambda$-unstable sequences as defined in \cite{RosseggerGonzalez} which are key to understanding the Scott sentence complexities near the limit levels.
	Note that the fourth criterion immediately guarantees that there are infinitely many isomorphism types among the collection $L_i$.
	This property can essentially be used as a substitute for the fourth criterion.
	
	\begin{lemma}\label{thin}
		If the sequence of linear orderings $\{L_i\}_{i\in\omega}$ has that for some non-zero fundamental sequence $\delta_i\to\lambda$:
		\begin{itemize}
			\item $\SR(L_i)<\lambda$
			\item $L_i\equiv_{\delta_i} L_{i+1}$
			\item there are infinitely many isomorphism types among the collection $L_i$,
		\end{itemize}
		then there is some increasing sequence $i(r)$ for which
		\begin{itemize}
			\item $\SR(L_{i(r)})<\lambda$
			\item $L_{i(r)}\equiv_{\delta_{i(r)}} L_{i(r+1)}$
			\item  $L_{i(r)}\not\equiv_{\delta_{i(r+1)}} L_{i(r+1)}$.
		\end{itemize}
	\end{lemma}
	
	\begin{proof}
		Inductively define such a sequence $i(r)$ as follows. The first two conditions are automatic, so we must satisfy the third condition.
		Let $i(0)=0$.
		Given $i(r)$, let $i(r+1)$ be the first index greater than $i(r)$ for which $L_{i(r)}\not\equiv_{\delta_{i(r+1)}} L_{i(r+1)}$. We note that such an index always exists.
		For the sake of contradiction say that for all $k>i(r)$,  $L_{i(r)}\equiv_{\delta_{k}} L_{k}$.
		Say $\SR(L_{i(r)})=\alpha<\lambda$.
		Fix some $n\geq i(r)$ with $\delta_n \geq \alpha+2$.
		By our assumption, for all $k\geq n$ we have that $L_{i(r)} \equiv_{\delta_{k}} L_{k}$ and thus $L_{i(r)}\cong L_{k}$.
		This is a contradiction to the fact that there are infinitely many isomorphism types among the collection $L_i$.
	\end{proof}
	
	This lemma is particularly useful when defining examples of $\lambda$-mixable pairs.
	So long as our example has that $\SR(L_i)<\lambda$, $L_i\equiv_{\delta_i} L_{i+1}$ and there are infinitely many isomorphism types among the collection $L_i$, we do not need to confirm that $L_{i(r)}\not\equiv_{\delta_{i(r+1)}} L_{i(r+1)}$, instead assuming that we pass to a suitable refinement with this property.
	In particular, we do not even need to confirm that there are $L_i$ with arbitrarily high Scott rank below $\lambda$; this is automatic from the other parts of the definition.
	We now take some space to consider examples in this vein.
	
	\begin{example}
		Fix $\delta_i\to\lambda$ and let $L_i\cong \omega^{\delta_i}$.
		The fact that $\SR(L_i)<\lambda$ and $L_i\equiv_{\delta_i} L_{i+1}$ were established by Ash \cite{As86}.
		Suitable choices for $K$ include $\zeta$ or $\sum_{i\in\eta} P(i)$ with $P(i)\in\{\zeta,\omega^*,\eta, \{\mathbf{k}\}_{k\in\omega}\}$.
		Unsuitable choices for $K$ include orderings such as $\eta\cdot\omega^{\delta_i}$ and $\omega$.
	\end{example}
	
	\begin{example}\label{zeta}
		Fix $\delta_i\to\lambda$ and let $L_i\cong \zeta^{\delta_i}$.
		The fact that $\SR(L_i)<\lambda$ and $L_i\equiv_{\delta_i} L_{i+1}$ were established by Gonzalez and Rossegger \cite[Lemma 3.5 and 3.8]{RosseggerGonzalez}.
		Suitable choices for $K$ include $\omega$ or $\sum_{i\in\eta} P(i)$ with $P(i)\in\{\omega,\omega^*,\eta, \{\mathbf{k}\}_{k\in\omega}\}$.
		Unsuitable choices for $K$ include orderings such as $\eta\cdot\zeta^{\delta_i}$ and $\zeta$.
	\end{example}
	
	The following example uses approximations to the linear ordering $L_\lambda=\sum_{n\in\omega} (n+\zeta^{\delta_n})$ with Scott complexity $\Pi_{\lambda}$ (see \cite[Proposition 3.10 and Theorem 3.11]{RosseggerGonzalez}) to obtain a suitable sequence of $L_i$.
	
	\begin{example}\label{piLambda}
		Fix $\delta_i\to\lambda$.
		Let $\L_i=\sum_{n< i} (n+\zeta^{\delta_n})+\sum_{n\geq i} (n+\zeta^{\delta_i}).$
		The fact that $\SR(L_i)<\lambda$ and $L_i\equiv_{\delta_i} L_{i+1}$ were established by Gonzalez and Rossegger \cite[Theorem 3.11]{RosseggerGonzalez}.
		$\omega$ is a suitable choice for $K$, but $\zeta$ is not.
	\end{example}
	
	With some intuition for $\lambda$-mixable pairs established with the above examples, we are now ready to define the construction that will have Scott complexity $\Sigma_{\lambda+1}$.
	
	\begin{definition}
		Given a $\lambda$-mixable pair $(\{L_i\}_{i\in\omega}, K)$ and a almost periodic function $f\colon\mathbb{Z}\to\mathbb{N}$, we let the \emph{$(f,\mathbb{Z})$-sum}, or simply \emph{$\mathbb{Z}$-sum}, of the pair be defined as:
		\[\sum_{n\in\mathbb{Z}} (1+ L_{f(n)} +1+K) =  \cdots + 1 + L_{f(-1)} + 1 + K + 1 + L_{f(0)} + 1 + K + 1 + L_{f(1)} + 1 + K + \cdots.\]
	\end{definition}
	
	\subsection{Almost periodic functions taking on infinite values}
	
	In order to define $\Z$-sums given functions with potentially infinite values, we need to extend our $\lambda$-mixable pairs to infinity as well.
	
	\begin{definition}
		Given a $\lambda$-mixable pair $(\{L_i\}_{i\in\omega}, K)$, we say that $\N$ is a \emph{limit structure} for the pair if for all $n$ there is an $i$ such that $\N\equiv_{\delta_n} L_i$.
	\end{definition}
	
	\begin{lemma}
		Every $\lambda$-mixable pair has a limit structure.
	\end{lemma}
	
	\begin{proof}
		We can construct the limit structure uniformly using Lemma XII.6 from \cite{Mo21}.
	\end{proof}
	
	It is possible that there are many different limit structures. But we are particularly interested in cases where there is a unique limit structure.
	
	\begin{definition}
		If a $\lambda$-mixable pair $(\{L_i\}_{i\in\omega}, K)$ has a \emph{unique limit structure} then we denote this structure by $L_\infty$.
	\end{definition}
	
	For example, consider a linear ordering $L_\infty$ of Scott complexity $\Pi_\lambda$ and let $L_i$ be an approximating sequence for this structure.
	An instance of such a sequence is given by Example \ref{piLambda}, which approximate $\sum_{n\in\omega} (n+\zeta^{\delta_n})$.
	The $L_i$ have a unique limit structure in this case as they fix the $\Pi_\lambda$ theory of a structure with Scott complexity $\Pi_\lambda$. In fact this type of situation is the only one with a unique limit structure.
	
	\begin{lemma}
		Given a $\lambda$-mixable pair $(\{L_i\}_{i\in\omega}, K)$, they have a unique limit structure $L_\infty$ if and only if they have a limit structure with Scott complexity $\Pi_\lambda$.
	\end{lemma}
	\begin{proof}
		Suppose that the $L_i$ have a limit structure $L_\infty$ with a $\Pi_\lambda$ Scott sentence. If $L'$ is any other limit structure, then $L_\infty \equiv_\lambda L'$, and since $L_\infty$ has a $\Pi_\lambda$ Scott sentence, $L_\infty \cong L'$.
		
		On the other hand, suppose that the $L_i$ have a unique limit structure $L_\infty$. Then $L_\infty$ is characterized among countable structures by the fact that for all $\alpha < \lambda$ there is $i$ such that $L_\infty \equiv_{\delta_i} L_i$. For each $\alpha$ this can be expressed by a $\Sigma_{< \lambda}$ sentence, and the conjunction of all of these is a $\Pi_\lambda$ Scott sentence. On the other hand, $L_\infty$ cannot have a simpler Scott sentence because for each $i$ we have $L_\infty \equiv_{\delta_i} L_i$ but $L_\infty \ncong L_i$.
	\end{proof} 
	
	We can also show that (5) of Definition \ref{defn:l-mixable} is inherited by the limit structures. We do this below in Lemma \ref{lem:inherit}.

	\begin{definition}
		Give a $\lambda$-mixable pair $(\{L_i\}_{i\in\omega},K)$ and a function $f\colon\Z\to\mathbb{N}\cup\{\infty\}$, the $(f,\Z)$-sums are the linear orders of the form
		$$\sum_{n\in\Z} (1+L^{(n)}_{f(n)} +1+K) $$
		where if $f(n) < \infty$ then $L^{(n)}_{f(n)} = L_{f(n)}$ and if $f(n) = \infty$ then $L^{(n)}_{f(n)}$ is \emph{some} limit structure for $(\{L_i\}_{i\in\omega},K)$, possibly different for each such $n$.
	\end{definition}
	
	If $f$ takes on only finite values, then this is just the $(f,\Z)$-sum as defined before. If $f$ takes on infinite values and $(\{L_i\}_{i\in\omega},K)$ does not have a unique limit structure, then there may be many $(f,\Z)$-sums (and these $(f,\Z)$-sums may not have $\Sigma_{\lambda + 1}$ Scott sentences). However if there is a unique limit structure, then there is only one $(f,\Z)$-sum.
	
	\begin{definition}
		If $(\{L_i\}_{i\in\omega},K)$ has a unique limit structure $L_\infty$, then \textit{the} $(f,\Z)$-sum of $(\{L_i\}_{i\in\omega},K)$ is
		$$\sum_{n\in\Z} (1+L_{f(n)} +1+K).$$ 
	\end{definition}
	
	
	%
	%
	
	\section{Useful Results about Linear Orderings}
	
	Before verifying the construction we will introduce several results about linear orders that we will use throughout the paper. We will make frequent use of the following standard result without mention.
	
	\begin{lemma}[{\cite[Lemma 15.7]{ash2000}}]\label{lem:bfandpartitions}  
		Let $\bar{a} \in A$ and $\bar{b} \in \B$ be tuples of length $n$. Then $(\A,\bar{a})\leq_\alpha (\B,\bar{b})$
		if and only if $(-\infty, a_1)\leq_\alpha (-\infty, b_1)$,
		$(a_n,\infty)\leq_\alpha (b_n,\infty)$, and for all $1 \leq i < n$ we have $(a_i,a_{i+1})\leq_\alpha (b_i,b_{i+1})$.
	\end{lemma}
	
	The general idea of the $(f,\Z)$-sums is that the delimiters $K$ are of Scott rank $< \lambda$, so they can be recovered with sentences of complexity $< \lambda$. (5) of Definition \ref{defn:l-mixable} says that there is a unique way to divide up an $(f,\Z)$-sum by its delimiters. We note that this definition is only applied to $\lambda$-mixable pairs and not their limit structures, so we begin by showing that it is inherited by limit structures.
	
	\begin{lemma}\label{lem:inherit}
		Let $(\{L_i\}_{i\in\omega}, K)$ be a $\lambda$-mixable pair. Consider a finite alternating sum $1+L+1+K+1+L+1+K+\cdots+1+L$ where each $L$ is one of the $L_i$ or a limit structure. Then this sum has intervals isomorphic to $K$ only within the written $K$ blocks (or as the entire written $K$ block).
	\end{lemma}
	\begin{proof}
		Let $\alpha = pSR(K) < \lambda$. Let $S$ be the finite alternating sum. Consider the alternating sum $T = 1+L_{a_0}+1+K+1+L_{a_1}+1+K+\cdots+1+L_{a_n}$ obtained by replacing each limit structure $L$ in the original sum by some $L_i \equiv_{\delta_i} L$ with $\delta_i > \alpha+2$. Let $\bar{s}$ be the finitely many delimiting elements marked ``1'' in $S$, and $\bar{t}$ the corresponding delimiting elements of $T$. Then by Lemma \ref{lem:bfandpartitions} we have  $(S,\bar{s}) \equiv_{\delta_i} (T,\bar{t})$. There is a $\Pi_{\alpha+2}$ formula expressing about $(T,\bar{t})$ that it satisfies (5) of Definition \ref{defn:l-mixable}, i.e., that any interval isomorphic to $K$ appears only within the written $K$ blocks. Since $\delta_i > \alpha + 2$, this formula is also true of $(S,\bar{s})$, and so any interval isomorphic to $K$ in $S$ appears only within the written $K$ blocks.
	\end{proof}
	
	%
	%
	We will use the following definability results about finding copies of the delimiter $K$ in an $(f,\Z)$-sum. 
	
	\begin{definition}
		Fix a countable linear ordering $K$.
		We write:
		$$D_K(x,y):= ~~ (x,y)\cong K \land \forall x', y' ~ (x'\leq x \land y\leq y') \to ((x=x' \land y=y') \lor (x',y')\not\cong K).$$ 
	\end{definition}
	
	We express $(x,y)\cong K$ using a Scott sentence for $K$, so that in particular, if $pSR(K)\leq\alpha$ then $D_K$ is $\text{d-}\Sigma_{\alpha+2}$ at worst. Intuitively, $D_K(x,y)$ says that $(x,y)$ is a copy of $K$ and it is maximal with regard to inclusion among such copies. The two "1"s surrounding the $K$ blocks in an $(f,\mathbb{Z})$-sum were specifically defined to have this exact property and no other points will have this property.
	
	\begin{lemma}\label{lem:n-sum}
		Let $pSR(K) \leq \alpha$. Then there is a $\Pi_{\alpha + 5}$ sentence $\varphi^{\mathbb{N}}_K$ such that a linear order $L$ satisfies $\varphi^{\mathbb{N}}_K$ if and only if there are linear orders $M_n$ such that
		\[ L \cong \sum_{n \in \mathbb{N}} (1 + M_n + 1 + K)\]
		and ($*$) any finite partial sum $1+M_{n}+1+K+1+M_{n+1}+1+K+\cdots+1+M_{n+k} + 1 + K$ has intervals isomorphic to $K$ only within the written $K$ blocks (or as the entire written $K$ block).
	\end{lemma}
	\begin{proof}
		$\varphi^{\mathbb{N}}_K$ says that for every $n$, there are elements $a_0 < b_0 < a_1 < b_1 < \cdots < a_n < b_n < a_{n+1}$ such that
		\begin{enumerate}
			\item $a_0$ is the minimal element,
			\item $D_K(b_i,a_{i+1})$  for each $i = 1,\ldots,n$, and
			\item for no other $x < y < a_{n+1}$ does $D_K(x,y)$ hold.
		\end{enumerate}
		Any $L \cong \sum_{n \in \mathbb{N}} (1 + M_n + 1 + K)$ with property ($*$) satisfies this sentence with $a_n$ and $b_n$ being the first and second $1$'s in the $n$th summand. The condition ($*$) means that (3) holds.
		
		On the other hand, if $\varphi^{\mathbb{N}}_K$ is true in a linear order $L$, then by (3) the choices of $a_0 < b_0 < a_1 < b_1 < \cdots$ are unique. Thus, letting $M_i = (b_i,a_i)$, we get
		\[ L \cong \sum_{n \in \mathbb{N}} (1 + M_n + 1 + K).\]
		Property (3) also yields ($*$).
	\end{proof}
	
	A similar argument works for $\mathbb{Z}$-sums as well. Here, we instead say that each pair $(x,y)$ satisfying $D_K(x,y)$ has a successor pair satisfying $D_K$ and a predecessor pair satisfying $D_K$, with no other pairs in between satisfying $D_K$.
	\begin{lemma}\label{lem:Z-sum}
		Let $SR(K) \leq \alpha$. Then there is a $\Pi_{\alpha + 5}$ sentence $\varphi^{\mathbb{Z}}_K$ such that a linear order $L$ satisfies $\varphi^{\mathbb{Z}}_K$ if and only if there are linear orders $M_n$ such that
		\[ L \cong \sum_{n \in \mathbb{Z}} (1 + M_n + 1 + K)\]
		and ($*$) any finite partial sum $1+M_{n}+1+K+1+M_{n+1}+1+K+\cdots+1+M_{n+k} + 1 + K$ has intervals isomorphic to $K$ only within the written $K$ blocks (or as the entire written $K$ block).
	\end{lemma}

	\begin{lemma}\label{lem:def-pieces}
		Let $pSR(K) \leq \alpha$ and let
		\[ L \cong \sum_{n \in \mathbb{N}} (1 + M_n + 1 + K)\]
		be a linear order satisfying the property ($*$) of Lemma \ref{lem:n-sum}.
		For each $n$, there is a $\Sigma_{\alpha + 4}$ formula $\psi_{n,K}(x)$ defining the subset $M_n$.
	\end{lemma}
	\begin{proof}
		$M_n$ is definable by the formula $\psi_{n,K}(x)$ which says that there are elements $a_0 < b_0 < a_1 < b_1 < \cdots < a_n < b_n < a_{n+1}$ such that
		\begin{enumerate}
			\item $a_0$ is the minimal element,
			\item $D_K(b_i,a_{i+1})$ for each $i = 1,\ldots,n$,
			\item for no other $x,y < a_{n+1}$ does $D_K(x,y)$ hold, and
			\item $a_n < x < b_n$.\qedhere
		\end{enumerate}
	\end{proof}
	
	\section{Verification of the Construction}
	
	In this section, we prove the following theorem. In particular, we obtain Theorem \ref{thm:main} which says that for any limit ordinal $\lambda$ there is a linear ordering $L_\lambda$ such that $SSC(L_\lambda)=\Sigma_{\lambda+1}$. Moreover, one can see from the construction that if $\lambda$ is computable and the $\lambda$-mixable pair is uniformly computable then $L_\lambda$ can be taken to be computable as well.
	
	\begin{theorem}\label{thm:sc-Ll}
		Given a $\lambda$-mixable pair $(\{L_i\}_{i\in\omega}, K)$ and an almost periodic function $f$, let $L_\lambda$ be the $\mathbb{Z}$-sum. If $f$ takes on only finite values or there is a unique limit structure $L_\infty$, then $L_\lambda$ has Scott complexity $\Sigma_{\lambda+1}$.
	\end{theorem}

	\begin{proof}		
		We fix the following notation: In $L_\lambda$, the first $1$ in the $n$th summand $1+ L_{f(n)} +1+K$ will be $a_n$, and the second $1$ in the $n$th summand will be $b_n$.
		
		\begin{claim}\label{claim:lower-bound}
			$L_\lambda$ does not have a $\Pi_{\lambda + 1}$ Scott sentence.
		\end{claim}
		\begin{proof}
			We show that $a_0$ does not have a $\Sigma_{\lambda}$ definable orbit. This is enough by the robustness theorem of Montalb\'an in \cite{Mo15}. Since $\lambda$ is a limit ordinal, this is equivalent to saying that for all $i$, $a_0$ does not have a $\Sigma_{\delta_i}$ definable orbit.
			In particular, we claim if $k$ is a multiple of the period of $\Res{f}{i}$, then $a_0 \equiv_{\delta_i} a_k$.
			Fix any $i$ and a corresponding $k$. By Lemma \ref{lem:bfandpartitions}, it suffices to show that $(-\infty,a_0) \equiv_{\delta_i} (-\infty,a_k)$ and $(a_0,\infty) \equiv_{\delta_i} (a_k,\infty)$.
			For the latter, we have
			$$(a_0,\infty) = L_{f(0)} + 1 + K + \sum_{n > 0}  (1+L_{f(n)}+1+K)$$
			and
			$$(a_k,\infty) = L_{f(k)} + 1 + K + \sum_{n > 0}  (1+L_{f(k+n)}+1+K).$$
			Since $k$ is a multiple of the period of $\Res{f}{i}$, for each $n$, $\Res{f}{i}(n)= \Res{f}{i}(k+n)$ and so $L_{f(n)} \equiv_{\delta_i} L_{f(k+n)}$.
			By playing the back and forth game on each corresponding summand, we get that $(a_0,\infty) \equiv_{\delta_i} (a_k,\infty)$. Symmetrically, we can argue that $(-\infty,a_0) \equiv_{\delta_i} (-\infty,a_k)$.
		\end{proof}
		
		\begin{claim}\label{claim:upper-bound}
			$L_\lambda$ has a $\Sigma_{\lambda+1}$ Scott sentence. 
		\end{claim}
		\begin{proof}
			We show that after adding the parameter $a_0$ to $L_\lambda$ it has a $\Pi_{\lambda}$ Scott sentence. The parameter cuts $L_\lambda$ into two intervals $(-\infty,a_0]$ and $[a_0,\infty)$, and it suffices to show that each of these has a $\Pi_\lambda$ Scott sentence. We consider only the latter, with a similar argument working for the former.
			
			Note that 
			\[ [a_0,\infty) = \sum_{n \geq 0}  (1+L_{f(n)}+1+K) \]
			A $\Pi_{\lambda}$ Scott sentence for this linear order is given as follows:
			\begin{enumerate}
				\item The sentence $\varphi^{\mathbb{N}}_K$ from Lemma \ref{lem:n-sum}, and
				\item for each $n$, the subset defined by the formula $\psi_{n,K}(x)$ from Lemma \ref{lem:def-pieces} is isomorphic to $L_{f(n)}$, as expressed by the Scott sentence of $L_{f(n)}$ (which is $\Pi_\lambda$ as either $f$ has no infinite values, or there is a unique limit structure $L_\infty$).
			\end{enumerate}
			
			Note that relativizing each of the quantifiers in the $\Pi_\lambda$ Scott sentence of $L_{f(n)}$ to the formula $\psi_{n,K}(x)$ yields a $\Pi_{\alpha+5+\lambda}=\Pi_{\lambda}$ sentence, as claimed.
			It is easy to see that $[a_0,\infty)$ satisfies this sentence.
			
			On the other hand, suppose that we have a linear order $M$ that satisfies these sentences. Then, by Lemma \ref{lem:n-sum},
			\[ M \cong \sum_{n \in \mathbb{N}} (1 + M_n + 1 + K)\]
			and by Lemma \ref{lem:def-pieces} each $M_n$ is the subset defined by $\psi_{n,K}(x)$ and hence, as it satisfies the Scott sentence of $L_{f(n)}$, is isomorphic to $L_{f(n)}$. Thus $M \cong [a_0,\infty)$.
		\end{proof}
		\renewcommand{\qedsymbol}{}
	\end{proof}
	
	The reader might wonder what happens if $f$ takes on at least one value $\infty$, and $(\{L_i\}_{i\in\omega}, K)$ does not have a unique limit structure, but nevertheless we fix some particular limit structure $L_\infty$ with respect to which we take the $(f,\Z)$-sum. Then (if $f$ has finitely many infinite values):
	\begin{itemize}
		\item If $SSC(L_\infty)=\Pi_{\lambda},\Sigma_{\lambda+1}$, then $SSC(L_{\lambda})=\Sigma_{\lambda+1}$.
		\item If $SSC(L_\infty)=\Pi_{\lambda+1},\text{d-}\Sigma_{\lambda+1}$, then $SSC(L_{\lambda})=\text{d-}\Sigma_{\lambda+1}$.
		\item Otherwise, $SSC(L_\lambda)=SSC(L_{\infty})$.
	\end{itemize}
	We omit the proofs since they are of a similar style to Theorem \ref{thm:sc-Ll}.
	
	The next two corollaries, stated previously in the introduction, are obtained by applying Theorem \ref{thm:sc-Ll} to particular $\lambda$-mixable pairs. First, recall that a structure is called \emph{rigid} if its automorphism group is trivial. 
	
	\rigid* 
	
	
	\begin{proof}
		Let $i\colon\mathbb{N}\to\eta$ be a bijective enumeration. Take 
		$$K=\sum_{q\in\eta} i^{-1}(q),$$
		where $i^{-1}(q)$ is the unique linear ordering with that many elements.
		Let $L_i:=\omega^{\delta_i}$.
		It is routine to check that these orderings are rigid and produce a $\lambda$-mixable pair.
		
		Construct $L_\lambda$ using this $\lambda$-mixable pair and any almost periodic function $f$:
		\[L_\lambda = \sum_{n\in\mathbb{Z}} (1+ L_{f(n)} +1+K).\]
		Any automorphism of $f$ must fix each written $K$ setwise, and hence also each $L_{f(n)}$. But these are all rigid, and so the only automorphism is the identity.
	\end{proof}
	
	Recall that a linear ordering is called \emph{scattered} if it does not accept an embedding from $\eta$.
	
	\scattered*
	
	\begin{proof}
		Take $K=\zeta$ and let  $L_i:=\omega^{\delta_i}$.
		It is routine to check that these orderings have the desired properties and that any associated $\mathbb{Z}$-sum is scattered.
	\end{proof}
	
	This is the first known case of a scattered linear ordering with a $\Sigma$ level Scott complexity.
	
	\begin{question}
		Are there scattered linear orderings with $SSC(L)=\Sigma_{\alpha}$ for $\alpha\neq\lambda+1$?
	\end{question}
	
	Alvir, Greenbereg, Harrison-Trainor, and Turetsky \cite{AGHT} showed that no structures have $\Sigma_{2}$ Scott complexity.
	Gonzalez and Rossegger \cite{RosseggerGonzalez} showed that no linear ordering has $\Sigma_{3}$ Scott complexity.
	We show that no scattered linear orderings have $\Sigma_{4}$ Scott complexity, though we delay the proof until the end of the paper since it would be a digression from the main techniques of this paper.

	\begin{proposition}\label{prop:no-scatter-sigma4}
		There are no scattered linear orderings with Scott complexity $\Sigma_{4}$. 
	\end{proposition}
	
	\section{Families of non-isomorphic $\lambda$-equivalent structures}\label{sec:many}
	
	In order to prove more sophisticated results about linear orderings with Scott complexity $\Sigma_{\lambda+1}$, we must delve further into the almost periodic functions that we use to construct the orderings.
	In this section we study the combinatorial structure of almost periodic functions.
	We also explain how to transfer facts about this combinatorial structure to facts about the Scott complexity of the associated linear orderings.
	
	We are generally dealing with multiple almost periodic functions at once in this section.
	For the sake of clarity, given a almost periodic function $f$, we will let $L_f$ denote the $\Z$-sum that corresponds to a given $\lambda$-mixable pair (that should be clear from context).

	We define two important equivalence relations among the almost periodic functions.
	We also demonstrate that these equivalence relations correspond to $\cong$ and $\equiv_\lambda$ in the corresponding linear orderings.
	
	\begin{definition}
		Given two almost periodic functions $f,g$ we write $f \EZ g$ to indicate that there is an $n\in\Z$ such that for all $m\in \Z$, we have $f(m)=g(m+n)$.
	\end{definition}
	
	\begin{proposition}\label{IsomorphismChar}
		For any fixed $\lambda$-mixable pair, $L_f\cong L_g$ if and only if $f \EZ g$.
	\end{proposition}
	
	\begin{proof}
		Assume that $f \EZ g$.
		In other words,  there is an $n\in\Z$ such that for all $m\in \Z$, $f(m)=g(m+n)$.
		Define a map from $L_f$ to $L_g$ summand by summand, sending the $m^{th}$ summand to the $(m+n)^{th}$ summand isomorphically.
		It is straightforward to confirm that this is an isomorphism.
		
		Now say that $L_f\cong L_g$ via the isomorphism $\varphi$.
		Note that $\varphi$ must preserve relation $D_K$ on pairs, and so must send every $K$ block in $L_f$ to a $K$ block in $L_g$.
		As the $K$ blocks have order type $\zeta$, the only way to do this is to, for some fixed $n\in\Z$, send the $m^{th}$ $K$ block in $L_f$ to the $(m+n)^{th}$ $K$ block in $L_g$.
		This means that $\varphi$ must restrict to an isomorphism $L_{f(m)}\cong L_{g(m+n)}$.
		As each element of the $L_i$ has a unique isomorphism type, this means that $f(m)=g(m+n)$, as desired.
	\end{proof}
	
	The other notion of equivalence we will consider is the coarser equivalence relation $\Efin$ defined as follows.
	
	\begin{definition}
		Given functions $f,g \colon \mathbb{Z}\to \mathbb{N} \cup \{\infty\}$ we write $f \Efin g$ to indicate that for all $\ell$ we have $\Res{f}{\ell} \EZ \Res{g}{\ell}$. 
	\end{definition}
	
	\noindent It is immediate that if $f$ is almost periodic and $f \Efin g$ then $g$ is also almost periodic.
	
	For almost periodic functions taking only finite values, there is also another characterization of $\Efin$. A \textit{finite successive segment} of $f$ is a sequence of consecutive values of $f$, e.g.,  $(f(n),f(n+1),\ldots,f(n+i))$ for some $n,i$.
	
	\begin{proposition}
		If $f$ and $g$ both take only finite values, then $f \Efin g$ if and only if $f$ and $g$ have the same finite successive segments.
	\end{proposition}
	\begin{proof}
		Suppose that $f \Efin g$. Given any finite successive segment $(f(n),f(n+1),\cdots,f(n+i))$ of $f$, let $\ell$ be greater than the maximal value of $f$ appearing there. Then $\Res{f}{\ell} \EZ \Res{g}{\ell}$, and so this finite successive segment also appears in $g$. By symmetry, any finite successive segment of $g$ also appears in $f$.
		
		On the other hand, suppose that $f$ and $g$ have the same finite successive segments. Given $\ell$, we must show that $\Res{f}{\ell} \EZ \Res{g}{\ell}$. But both $\Res{f}{\ell}$ and $\Res{g}{\ell}$ are periodic, and periodic functions are determined up to $\EZ$-equivalence by their finite successive segments. (If function is not periodic with period $\pi$, then there is a finite successive segment witnessing this. Thus the finite successive segments determine the period, and then a segment of the same size as the period determines the function.)
	\end{proof}
	
	Using the fact that periodic functions are determined by their finite successive segments, we get the following for functions that might take on infinite values.
	
	\begin{proposition}
		Given functions $f,g \colon \mathbb{Z}\to \mathbb{N} \cup \{\infty\}$, $f \Efin g$ if and only if for all $n$ the restrictions $\Res{f}{n}$ and $\Res{g}{n}$ have the same finite successive segments.
	\end{proposition}
	
	The next two propositions show that the relation $\Efin$ between almost periodic functions is related to the equivalence  $\equiv_\lambda$ for the corresponding $\mathbb{Z}$-sums.
	
	\begin{proposition}\label{prop:transfer}
		Consider two almost periodic functions $f$ and $g$ potentially with infinite values and a fixed $\lambda$-mixable pair. Let $L_f$ and $L_g$ be any $(f,\Z)$- and $(g,\Z)$-sums respectively. Then 
		$L_f\equiv_\lambda L_g$ if and only if $f \Efin g$.
	\end{proposition}
	
	\begin{proof}
		We abuse notation by writing, when $f(n) = \infty$, $L_{f(n)}$ for the particular limit structure in the $n$th summand of $L_f$.
		
		For the backward direction, consider an arbitrary $\lambda$-mixable pair $(\{L_i\},K)$ and let $\delta_i$ be the fundamental sequence associated to this pair.
		It is enough to show that for all $n$ $L_f\equiv_{\delta_n} L_g$.
		Fix $n$ and let $a$ be such that for all $m$ $\Res{f}{n}(m+a)=\Res{g}{n}(m)$.
		We show that $L_f\equiv_{\delta_n} L_g$ by looking summand by summand.
		In particular, we compare the $(m+a)^{th}$ summand of $L_f$ with the $m^{th}$ summand of $L_g$.
		We need only show that $L_{f(m+a)}\equiv_{\delta_n} L_{f(m)}$.
		However, this follows at once from the fact that $\Res{f}{n}(m+a)=\Res{g}{n}(m)$.
		
		We now show the forward direction. For each $\ell,m,n$, we want to show that the finite successive segment $\Res{f}{n}(m),\ldots,\Res{f}{n}(m+\ell-1)$ is also a finite successive segment of $\Res{g}{n}$. Consider the sentence which says that there are $x_0 < y_0 < \ldots < x_\ell < y_\ell$ such that
		\begin{enumerate}
			\item $D_K(x_i,y_i)$ for all $i$ and all other $x,y$ with $x_0 \leq x,y \leq y_\ell$ does not satsfiy $D_K(x,y)$, and
			\item The interval $(y_i,x_{i+1})$ is $\equiv_{\delta_n}$-equivalent to $L_{f(m+i)}$.
		\end{enumerate} 
		Informally, the above formula states that there is an interval of type
		\[ 1 + K + 1 + M_m + 1 + K + 1 + M_{m+1} + 1 + \cdots + 1 + K + 1 + M_{m+\ell-1} + 1 + K + 1 \]
		with $M_{m+i} \equiv_{\delta_n} L_{f(m+i)}$ for $0 \le i < \ell$.
		This is a $\Pi_{< \lambda}$ sentence, and holds of $L_f$. Thus it must also hold of $L_g$. In particular, there is some $m'$ such that for $0 \leq i < \ell$, $L_{g(m'+i)} \equiv_{\delta_n} L_{f(m+i)}$. This implies that $\Res{g}{n}(m'+i) = \Res{f}{n}(m+i)$.
		
		Since for a fixed $n$ this is true for every $\ell,m$, every finite successive segment of $\Res{f}{n}$ is a finite successive segment of $\Res{g}{n}$, and a similar argument in the other direction shows that every finite successive segment of $\Res{g}{n}$ is a finite successive segment of $\Res{f}{n}$. Moreover, this is true for all $n$, and hence $f \Efin g$.
	\end{proof}
	
	\begin{proposition}
		If $L_f$ is a $\Z$-sum with respect to a $\lambda$-mixable pair $(\{L_i\}_{i\in\omega},K)$ and $M\equiv_\lambda L_f$, then $M$ is a $(g,\Z)$-sum for some $g \Efin f$ with potentially infinite values. In particular, if $(\{L_i\}_{i\in\omega},K)$ has a unique limit $L_\infty$, then $M \cong L_g$.
	\end{proposition}
	
	\begin{proof}
		Since $L_f$ satisfies the sentence $\varphi^{\mathbb{Z}}_K$ from Lemma \ref{lem:Z-sum}, $M$ must also satisfy it, and so there is a unique way (up to a shift) to write
		\[ M = \sum_{n \in \mathbb{Z}} (1 + M_n + 1 + K)\]
		where ($*$) any finite partial sum $1+M_{n}+1+K+1+M_{n+1}+1+K+\cdots+1+M_{n+k} + 1 + K$ has intervals isomorphic to $K$ only within the written $K$ blocks (or as the entire written $K$ block).
		
		Define $\theta(x,y)$ if there are $x' < x$ and $y' > y$ such that $D_K(x',x)$ and $D_K(y,y')$ but $D_K$ does not hold of any $u,v \in (x,y)$. Then $\theta(x,y)$ defines the endpoints of the intervals $1 + M_n + 1$ in $M$ and $1 + L_{f(n)} + 1$ in $L_f$.
		
		$L_f$ satisfies the $\Pi_{\lambda}$ sentence which says that for all $x,y$ if $\theta(x,y)$ then for every $n$ we have $(x,y) \equiv_{\delta_{n+1}} L_{n+1}$ or $(x,y) \cong L_0$ or $(x,y) \cong L_1$ or \ldots or $(x,y) \cong L_n$. We express isomorphism to $L_i$ using the Scott sentence for $L_i$. Since this sentence is true in $L_f$, it is also true in $M$. 
		
		Thus each $M_n$ is either isomorphic to some $L_i$, or $M_n 
		\equiv_{\delta_i} L_i$ for all $i$. The latter implies that $M_n$ is a limit structure of $(\{L_i\},K)$. Define $g \colon \mathbb{Z} \to \mathbb{N}\cup \{\infty\}$ such that $M_n = L_{g(n)}$ (or $g(n) = \infty$ if $M_n$ is a limit structure). Thus $M$ is a $(g,\Z)$-sum. The fact that $g \Efin f$ follows from Proposition \ref{prop:transfer}.
		
		For the final clause, if $(\{L_i\},K)$ has a unique limit structure $L_\infty$, then for all $n$, $M_n \cong L_{g(n)}$, and $M$ is the (unique) $(g,\Z)$-sum $L_g$.
	\end{proof}
	
	Let $L_\lambda$ be a $\mathbb{Z}$-sum with respect to a $\lambda$-mixable pair with unique limit structure and some almost periodic function $f$. Then $L_\lambda$, and all structures $M \equiv_\lambda L_\lambda$, are $\mathbb{Z}$-sums with respect this same $\lambda$-mixable pair (but possibly different almost periodic functions $g \Efin f$) and so have Scott sentence complexity $\Sigma_{\lambda + 1}$. This is almost Theorem \ref{thm:main2} except that we need to construct many different $g \Efin f$. This is our next goal.
	
	\section{Examples of almost periodic Functions}
	
	A large class of almost periodic functions will come from shifts of the $p$-adic valuations. Fix a prime $p$. Let $\mathbb{Z}_p$ be the $p$-adic integers, and recall that they are the inverse limit $\mathbb{Z}_p:= \varprojlim_n \mathbb{Z} / p^n$. Let $\pi_n : \mathbb{Z}_p \to \mathbb{Z} / p^n$ be the standard projections. Let $v_p \colon \mathbb{Z}_p \to \mathbb{N} \cup \{\infty\}$ be the $p$-adic valuation. For simplicity, we write $v$ for $v_p$ with the fixed prime $p$ understood.
	
	\begin{definition}
		Given $t \in \mathbb{Z}_p$, define $f_t \colon \mathbb{Z} \to \mathbb{N} \cup \{\infty\}$ by $f_t(k) = v(k-t)$.
	\end{definition}
	
	If the reader prefers not to think $p$-adically, an alternate definition of $f_t$ can be seen in Theorem \ref{thm:E0-red}, and all of the following can be proved directly.
	
	\begin{lemma}
		For all $t \in \mathbb{Z}_p$, $f_t$ is an almost periodic function. 
	\end{lemma}
	\begin{proof}
		First we show that $f_t$ is almost periodic. This follows from the fact that $v(x+y) \geq \min(v(x),v(y))$ with equality if $v(x) \neq v(y)$. Given $\ell$, for all $k$ we have
		\[ f_t(k+p^\ell) = v(k + p^\ell - t) \geq \min(v(k-t),\ell) = \min(f_t(k),\ell). \]
		If $f_t(k) < \ell$, then we have equality, $f_t(k+p^\ell) = f_t(k)$, and if $f_t(k) \geq \ell$, then we also have $f_t(k+p^\ell) \geq \ell$.
		
		To see that $f_t$ is not periodic, it suffices to note that $f_t(k) = v(k-t)$ takes unbounded values in $\mathbb{N}$. If $t \in \mathbb{Z}$, then $f_t(t + p^\ell) = v(p^\ell) = \ell$ for sufficiently large $\ell$. Otherwise, if $t \notin \mathbb{Z}$, write $t$ as a $p$-adic series $t = \sum_{i=0}^\infty a_i p^i$. Consider the partial sums $t_n = \sum_{i=0}^n a_i p^i \in \mathbb{Z}$. Then $f_t(t_n) = v(t_n - t) > n$.
	\end{proof}
	
	\begin{lemma}\label{lem:efin-equiv}
		Given  $s,t \in \mathbb{Z}_p$, $f_s \Efin f_t$.
	\end{lemma}
	\begin{proof}
		Fix $\ell$. Let $s = \sum_{i=0}^\infty a_i p^i$ and $t = \sum_{i=0}^\infty b_i p^i$ be $p$-adic series for $s$ and $t$. Let $n =  \sum_{i=0}^\ell (a_i - b_i) p^i$. Note that $v(t-s + n) > \ell$. Then for any $k\in\mathbb{Z},$
		\[ v(k+ n - s) = v(k - t + (t-s+n) ) \geq \min(v(k-t),v(t-s+n))\]
		with equality if $v(k-t) \leq \ell$. Thus
		\[ f_s^{\leq \ell}(k + n) = f_t^{\leq \ell}(k).\qedhere\]
	\end{proof}
	
	\begin{lemma}\label{lem:diff}
		Given $s,t \in \mathbb{Z}_p$, $s-t \in \mathbb{Z}$ if and only if $f_s \EZ f_t$.
	\end{lemma}
	\begin{proof}
		If $s-t = n \in \mathbb{Z}$, then
		\[ f_s(k+n) = v(k+n-s) = v(k-t) = f_t(k).\]
		On the other hand, suppose that $f_s \EZ f_t$ so that for some $n$, $f_s(k+n) = f_t(k)$. Then we claim that $s-t = n$. Indeed, given $t_i \to t$ an integer approximation of $t$, we must have
		\[ f_t(t_i) = v(t_i -t) \to \infty.\]
		But we also have $f_t(t_i) = f_s(t_i + n)$ and so
		\[ f_s(t_i + n) = v(t_i + n - s) \to \infty\]
		so that $t_i + n \to s$ is an integer approximation of $s$. Thus $s = t + n$.
	\end{proof}
	
	\begin{proposition}
		Given $t \in \mathbb{Z}_p$, $f_t$ takes on infinite values if and only if $t \in \mathbb{Z}$, in which case $f_t(t) = \infty$ is the only infinite value of $f_t$.
	\end{proposition}
	\begin{proof}
		This is because $v(x) = \infty$ if and only if $x = 0$.
	\end{proof}
	
	\begin{theorem}\label{thm:E0-red}
		There is a Borel reduction $E_0 \leq \EZ$ with the image contained inside a single $\Efin$ class. In particular, there is an $\Efin$ class with continuum many non-$\EZ$-equivalent almost periodic functions.
	\end{theorem}
	\begin{proof}
		Given $\bar{a} = (a_0,a_1,a_2,\ldots) \in 2^{\omega}$, associate with $\bar{a}$ the corresponding $t_a \in \mathbb{Z}_2$ with $\bar{a}$ as the coefficients of its 2-adic series, namely,
		\[ t_a = \sum_{i=0}^\infty a_i 2^i.\]
		Note that $\bar{a} \; E_0 \; \bar{b}$ if and only if $t_a - t_b \in \mathbb{Z}$. Thus (using Lemma \ref{lem:diff}) $a \mapsto f_{t_a}$ is the desired reduction. By Lemma \ref{lem:efin-equiv}, all $f_{t_a}$ lie in the same $\Efin$ class. 
		
		To see that this is Borel, note that $f_{t_{\bar{a}}}(k)$ may equivalently be defined directly from $\bar{a}$ as the greatest $\ell$ such that
		\[ k \equiv a_0 2^0 + a_1 2^1 + a_2 2^2 + \cdots + a_{\ell-1} 2^{\ell-1} \pmod{2^\ell}.\]
		If this is true for all $\ell$, then $f_a(k) = \infty$.
	\end{proof}
	
	\maintwo*
	
	\begin{proof}
		Take a $\lambda$-mixable sequence $(\{L_i\},K)$ with unique limit structure $L_\infty$, and take $L_\lambda$ to be the $(f,\Z)$-sum with respect to $f = f_t$ for some $t \in \mathbb{Z}_2$. Then, by previous results:
		\begin{enumerate}
			\item the structure $\equiv_\lambda$-equivalent to $L_\lambda$ are exactly the $(g,\Z)$-sums $L_g$ for some $g \Efin f$;
			\item each $L_g$ has Scott complexity $\Sigma_{\lambda+1}$;
			\item $L_g \cong L_h$ if and only if $g \EZ h$; and
			\item there are continuum many $\Efin$-equivalent but non-$\EZ$-equivalent functions.
		\end{enumerate}
		Taken all together, the theorem is proved.
	\end{proof}
	
	An interesting feature of this theorem is that the ``dual'' form is false. To show this, we will make essential use of the following definition and theorem of Gonzalez and Rossegger \cite{RosseggerGonzalez}.
	
	\begin{definition}
		For a structure $\A$ and a limit ordinal $\lambda$, a
		\emph{$\lambda$-sequence} in $\A$ is a set of tuples $\bar{y}_i\in \A$ for
		$i\in\omega$ such that $\bar{y}_i\equiv_{\alpha_i}\bar{y}_{i+1}$ for some fundamental
		sequence $(\alpha_i)_{i\in\omega}$ for $\lambda$.  We say that a $\lambda$-sequence is \emph{unstable} if $\bar{y}_i\not\equiv_{\alpha_{i+1}}\bar{y}_{i+1}$.
	\end{definition}
	
	\begin{theorem}[Gonzalez and Rossegger \cite{RosseggerGonzalez}]\label{unstable}
		Let $\A$ be a structure with $uSR(\A)=\lambda$ for $\lambda$ a limit ordinal. Then the Scott sentence complexity
		of $\A$ is $\Pinf{\lambda}$ if and only if there are no unstable
		$\lambda$-sequences in $\A$.
	\end{theorem}
	
	\nopistr*
	
	
	\begin{proof}
		Fix $M$ with Scott complexity $\Pi_{\lambda+1}$.
		By Rossegger and Gonzalez, this structure has a $\lambda$-unstable sequence $\{a_i\}_{i\in\omega}$.
		More explicitly, there is a fundamental sequence $\delta_i\to\lambda$ such that for all $i$, $a_i\equiv_{\delta_i} a_{i+1}$ yet $a_i\not\equiv_{\delta_{i+1}} a_{i+1}$.
		As structures we can observe the sequence $(M,a_0)\equiv_{\delta_0} (M,a_1)\equiv_{\delta_1} (M,a_2)\equiv_{\delta_2}\cdots$.
		Reindexing the sequence if needed, we can apply Lemma XII.6 from \cite{Mo21} to obtain a structure $(N,b)$ such that for all $i$, $(N,b)\equiv_{\delta_i} (M,a_i)$.
		Fix $i$ and let $k$ be such that $2\delta_{i+1}+1\leq \delta_k$.
		Note that
		\[ (M,a_k)\models \exists c ~ a_k\equiv_{\delta_i} c \land a_k\not\equiv_{\delta_{i+1}} c.\]
		This formula is expressible in $2\delta_{i+1}+1\leq \delta_k$ many quantifiers by Lemma VI.14 from \cite{Mo21}.
		This means that for all $i$ there is a $c_i$ such that $b\equiv_{\delta_i} c_i \land b\not\equiv_{\delta_{i+1}} c_i$.
		In other words, $b$ cannot have a $\Sigma_\lambda$ definable automorphism orbit inside of $N$.
		This means that $uSR(N)>\lambda$.
		In particular, $N\equiv_\lambda M$ yet the Scott complexity of $M$ is not $\Pi_{\lambda+1}$.
	\end{proof}
	
	This means that creating structures with exactly $\Pi_{\lambda+1}$ Scott complexity in the style of \cite{HT18} is impossible.
	
	\section{Proof of Proposition \ref{prop:no-scatter-sigma4}}
	
	We still have an outstanding debt, namely the proof of Proposition \ref{prop:no-scatter-sigma4}. We must show that there are no scattered linear orderings of Scott complexity $\Sigma_4$.
	
	\begin{proof}[Proof of Proposition \ref{prop:no-scatter-sigma4}]
		A linear ordering $\N$ with a $\Sigma_{4}$ Scott sentence must be of the form $L_1+1+L_2+1+\cdots+1+L_n$ where the $L_i$ have a $\Pi_{3}$ Scott sentence.
		Furthermore, each of these $L_i$ must be scattered, or else the whole linear ordering is not scattered.
		It follows from \cite{RosseggerGonzalez} that this means each of the $L_i$ must be among $\{\mathbf{k}\}_{k\in\omega},\zeta,\omega,\omega^*$ and $\omega+\omega^*$.
		
		Recall that for $a, b$ in a linear ordering, we say $a\sim_1 b$ if the interval $[a,b]$ is finite. By analyzing all possible cases, note that $\N$ must have the form $K_1+\cdots+K_m$ with each of the $K_i$ among  $\{\mathbf{k}\}_{k\in\omega},\zeta,\omega$ and $\omega^*$, and each $K_i$ is a $\sim_1$-equivalence class.
		In particular, $\N/\sim_1$ is finite. 
		
		We choose parameters in $L_i$ in each $\sim_1$ class and analyze the complexity of the parameters.
		The parameters we choose depends on the isomorphism type of the $\sim_1$ class.
		\begin{enumerate}
			\item If the $\sim_1$ class is finite, we take every element.
			\item If the $\sim_1$ class is $\omega$ we take the first element.
			\item If the $\sim_1$ class is $\omega^*$ we take the last element.
			\item If the $\sim_1$ class is $\zeta$ we take any element.
		\end{enumerate}
		
		Given $\N=K_1+\cdots+K_m$ say we choose $\bar{p}=p_1<\cdots<p_k$ according to the above rules.
		By convention let $p_0=-\infty$ and $p_{k+1}=\infty$.
		Checking case by case, it is easily confirmed that for all $i\leq k$, $(p_i,p_{i+1})$ is either empty or it is isomorphic to $\omega$, $\omega^*$ or $\omega+\omega^*$.
		In other words, it is immediate that $(\N,\bar{p})$ has a $\Pi_{3}$ Scott sentence.
		
		We now bound the complexity of the automorphism orbit of $\bar{p}$.
		Note that $\sim_1$ is a $\Sigma_{2}$ binary predicate and that the successor relation $S$ is a $\Pi_{1}$ binary predicate and define the following $\Pi_{2}$ formula:
		$$
		\psi(\bar{x}):= ~~~ 
		\bigwedge_{\{(k,j)\vert [p_k,p_j]\cong [p_k]_{\sim_1}\}}\big(\bigwedge_{k\leq i<j} S(x_{i},x_{i+1}) \big) \land \forall z<x_k\lnot S(z,x_k) \land  \forall z>x_j\lnot S(x_j,z)$$
		$$\land \bigwedge_{\{i\vert [p_i]_{\sim_1}\cong\omega\}}\forall z<x_i\lnot S(z,x_i)
		\land \bigwedge_{\{i\vert [p_i]_{\sim_1}\cong\omega^*\}}\forall z>x_i\lnot S(x_i,z)
		\land \bigwedge_{\{i\vert p_i\not\sim_1 p_{i+1}\}}  x_i\not\sim_1 x_{i+1} .$$
		Translating directly, the formula $\psi(\bar{x})$ states that we have picked $\bar{x}$ according to rules 1-4 enumerated above.
		
		We argue that $\psi(\bar{x})$ defines the orbit of $\bar{p}$, i.e., that rules 1-4 define the automorphism orbit of $\bar{p}$.
		Note that any automorphism of $\N$ descends to a necessarily trivial automorphism of the rigid $\N/\sim_1$.
		In other words, every automorphism factors into automorphisms of the the $\sim_1$ blocks.
		Therefore, the only non-trivial action of an automorphism shifts the $\zeta$ blocks as all other $\sim_1$ blocks are rigid.
		
		Thus $\psi(\bar{x})$ defines the automorphism orbit of $\bar{p}$.
		Because $\psi$ is $\Pi_{2}$, we get that there must be a $\text{d-}\Sigma_{3}$ Scott sentence for $\N$.
		This means that every scattered linear ordering with a $\Sigma_{4}$ Scott sentence has a $\text{d-}\Sigma_{3}$ Scott sentence, and there are no scattered linear orderings with Scott complexity $\Sigma_{4}$.
	\end{proof}


\begin{thebibliography}{}

\bibitem[AGHTT21]{AGHT}
Rachael Alvir, Noam Greenberg, Matthew Harrison-Trainor, and Dan Turetsky.
\newblock Scott complexity of countable structures.
\newblock {\em J. Symb. Log.}, 86(4):1706--1720, 2021.

\bibitem[AK]{ash2000}
Chris Ash and Julia Knight.
\newblock {\em Computable Structures and the Hyperarithmetical Hierarchy},
  volume 144.
\newblock {Newnes}.

\bibitem[Ash86]{As86}
C.~J. Ash.
\newblock Recursive labelling systems and stability of recursive structures in
  hyperarithmetical degrees.
\newblock {\em Trans. Amer. Math. Soc.}, 298(2):497--514, 1986.

\bibitem[Boh25]{Bohr}
Harald Bohr.
\newblock {Zur Theorie der Fastperiodischen Funktionen: II. Zusammenhang der
  fastperiodischen Funktionen mit Funktionen von unendlich vielen Variabeln;
  gleichmässige Approximation durch trigonometrische Summen}.
\newblock {\em Acta Mathematica}, 46(1-2):101 -- 214, 1925.

\bibitem[GR23]{RosseggerGonzalez}
David Gonzalez and Dino Rossegger.
\newblock Scott sentence complexities of linear orderings.
\newblock {\em arXiv preprint arXiv:2305.07126}, 2023.

\bibitem[HT18]{HT18}
Matthew Harrison-Trainor.
\newblock Scott ranks of models of a theory.
\newblock {\em Adv. Math.}, 330:109--147, 2018.

\bibitem[LE65]{Lo65}
E.~G.~K. Lopez-Escobar.
\newblock An interpolation theorem for denumerably long formulas.
\newblock {\em Fund. Math.}, 57:253--272, 1965.

\bibitem[Mil83]{Mi83}
Arnold~W. Miller.
\newblock On the {B}orel classification of the isomorphism class of a countable
  model.
\newblock {\em Notre Dame J. Formal Logic}, 24(1):22--34, 1983.

\bibitem[Mon15]{Mo15}
Antonio Montalb\'{a}n.
\newblock A robuster {S}cott rank.
\newblock {\em Proc. Amer. Math. Soc.}, 143(12):5427--5436, 2015.

\bibitem[Mon21]{Mo21}
Antonio Montalb\'{a}n.
\newblock {\em Computable structure theory---within the arithmetic}.
\newblock Perspectives in Logic. Cambridge University Press, Cambridge;
  Association for Symbolic Logic, Ithaca, NY, 2021.

\bibitem[Sco65]{Sc65}
Dana Scott.
\newblock Logic with denumerably long formulas and finite strings of
  quantifiers.
\newblock In {\em Theory of {M}odels ({P}roc. 1963 {I}nternat. {S}ympos.
  {B}erkeley)}, pages 329--341. North-Holland, Amsterdam, 1965.

\bibitem[Ste78]{St78}
John~R. Steel.
\newblock On {V}aught's conjecture.
\newblock In {\em Cabal {S}eminar 76--77 ({P}roc. {C}altech-{UCLA} {L}ogic
  {S}em., 1976--77)}, volume 689 of {\em Lecture Notes in Math.}, pages
  193--208. Springer, Berlin, 1978.

\bibitem[Tur20]{Turetsky20}
Dan Turetsky.
\newblock Coding in the automorphism group of a computably categorical
  structure.
\newblock {\em J. Math. Log.}, 20(3):2050016, 24, 2020.

\end{thebibliography}
\end{document}